\documentclass{amsart}

\usepackage{graphicx}
\usepackage{mathtools}
\usepackage{mathrsfs}
\usepackage{color}
\usepackage{soul}

\newtheorem{theorem}{Theorem}[section]
\newtheorem{lemma}[theorem]{Lemma}

\theoremstyle{definition}
\newtheorem{definition}[theorem]{Definition}

\newtheorem{remark}[theorem]{Remark}
\theoremstyle{corollary}
\newtheorem{corollary}[theorem]{Corollary}

\theoremstyle{proposition}
\newtheorem{proposition}[theorem]{Proposition}

\theoremstyle{conjecture}
\newtheorem{conjecture}[theorem]{Conjecture}

\theoremstyle{conditionalproposition}
\newtheorem{conditionalproposition}[theorem]{Conditional Proposition}

\numberwithin{equation}{section}

\begin{document}

\title[Double, borderline, and extraordinary eigenvalues]{Double, borderline, and extraordinary eigenvalues of Kac--Murdock--{Szeg\H{o}}  matrices\\ with a complex parameter}

\author{George Fikioris}

\address{School of Electrical and Computer engineering, National Technical University of Athens, GR 157-73 Zografou, Athens, Greece}
\email{gfiki@ece.ntua.gr}

\author{Themistoklis K. Mavrogordatos}
\address{Department of Physics, AlbaNova University Center, SE 106 91, Stockholm, Sweden}
\email{themis.mavrogordatos@fysik.su.se}


\date{\today}


\begin{abstract}
For all sufficiently large complex {$\rho$}, and for arbitrary matrix dimension $n$, it is shown that the  Kac–-Murdock–-{Szeg\H{o}} matrix $K_n(\rho)=\left[\rho^{|j-k|}\right]_{j,k=1}^{n}$ possesses exactly two eigenvalues whose magnitude is larger than $n$.  We discuss a number of properties of the two ``extraordinary'' eigenvalues. Conditions are developed that, given $n$, allow us---without actually computing eigenvalues---to find all values $\rho$ that give rise to eigenvalues of magnitude $n$, termed ``borderline'' eigenvalues. The aforementioned values of $\rho$ form two closed curves in the complex-$\rho$ plane. We describe these curves, which are $n$-dependent, in detail.
An interesting borderline case arises when an eigenvalue of $K_n(\rho)$ equals $-n$: apart from certain exceptional cases, this occurs if and only if the eigenvalue is a double one; and if and only if the point $\rho$ is a cusp-like singularity of one of the two closed curves. 
\end{abstract}

\maketitle

\section{Introduction}

This paper is a direct continuation of \cite{Fik2018}. For $\rho \in \mathbb{C}$, it deals with the complex-symmetric Toeplitz matrix 
\begin{equation}\label{matrixdefinition}
K_{n}(\rho)=\left[\rho^{|j-k|}\right]_{j,k=1}^{n}=\begin{bmatrix}
1 & \rho & \rho^2 & \ldots & \rho^{n-1} \\
\rho & 1 & \rho & \ldots & \rho^{n-2} \\
\rho^2 & \rho & 1 & \ldots & \rho^{n-3} \\
\vdots & \vdots & \vdots & \ddots & \vdots \\
\rho^{n-1} & \rho^{n-2} & \rho^{n-3} & \ldots & 1
\end{bmatrix}.
\end{equation}

For the special case $0<\rho<1$, $K_n(\rho)$ is often called the Kac–-Murdock–-{Szeg\H{o}} matrix; see \cite{Fik2018} for a history and a discussion of applications. Assuming throughout that $n=2,3,\ldots$, we categorize the eigenvalues of $K_n(\rho)$ as {follows.}

\begin{definition}\label{def:borddef1}
Let $\rho \in \mathbb{C}$. An eigenvalue $\lambda \in \mathbb{C}$ of $K_{n}(\rho)$ is called {\it ordinary} if $|\lambda| \leq n$ and {\it extraordinary} if $|\lambda|>n$.
\end{definition}

The terms ordinary/extraordinary are consistent with Section 6 of \cite{Fik2018},\footnote{They are specifically consistent with Corollary 6.9 of \cite{Fik2018} which, for $\rho \in \mathbb{R}\setminus\{-1,1\}$, is a necessary and sufficient condition for an eigenvalue to be extraordinary. {Ref.} \cite{Fik2018} actually defines extraordinary eigenvalues differently, see Definition 6.2 of \cite{Fik2018}. That definition, however, appears to be irrelevant to the more general case $\rho \in \mathbb{C}$.} the investigation of which pertains only to the  special case $\rho \in \mathbb{R}$ (this is when $K_n(\rho)$ is a real-symmetric matrix). Ref. \cite{Fik2018} further discusses a number of connections between extraordinary eigenvalues and the notions of wild, outlying, and un-{Szeg\H{o}}-like eigenvalues developed by Trench \cite{Trench1}, \cite{Trench2}, \cite{Bottcher}.

\begin{definition}\label{def:borddef2}
Let $\rho \in \mathbb{C}$. An eigenvalue $\lambda \in \mathbb{C}$ of $K_{n}(\rho)$ is called {\it borderline} if $|\lambda|=n$ {and} {\it{ borderline/double}} {if} $\lambda$ {is a borderline eigenvalue and, concurrently, an (algebraically) double eigenvalue.}
\end{definition}

 Our main goal herein is to find the $\rho$ that give rise to borderline and extraordinary eigenvalues (the said values of $\rho$ are, of course, $n$-dependent).  In other words, this paper extends the concept of extraordinary eigenvalues from the case $\rho \in \mathbb{R}$ to the case  $\rho \in \mathbb{C}$, and develops conditions for the occurrence of extraordinary eigenvalues. 

By Theorems 3.7 and 4.1 of \cite{Fik2018} the eigenvalues of $K_{n}(\rho)$ ($\rho \in \mathbb{C}$) are either of \textit{type-1} or \textit{type-2}. The two types are mutually exclusive. The most  notable distinguishing feature is that type-1 (type-2) eigenvalues correspond to skew-symmetric (symmetric) eigenvectors, see Remark 4.3 of \cite{Fik2018}. 

In the special case $\rho \in \mathbb{R}$, there are at most two extraordinary eigenvalues---one of each type---irrespective of how large $n$ is. In fact, using the notation
\begin{equation}\label{eq:xin}
\xi_n=\frac{n+1}{n-1},
\end{equation}
it readily follows from the results in Section 6 of \cite{Fik2018} and from Lemma 2.5 of \cite{Fik2018} (that lemma is repeated as Lemma \ref{lemma:lemma1} below) {that the following holds.}

\begin{proposition}\label{prop:realcase}  \cite{Fik2018} Let $\rho \in \mathbb{R}$. The matrix $K_n(\rho)$ possesses an eigenvalue $\lambda$ that is borderline iff $\rho\in\{-\xi_n,-1,1,\xi_n\}$. The specific value and type is as follows,

\noindent (i) $\rho=\xi_n$: $\lambda=-n$ is a type-1 borderline eigenvalue.

\noindent (ii) $\rho=1$: $\lambda=n$ is a type-2 borderline eigenvalue.

\noindent (iii) $\rho=-1$: $\lambda=n$ is a borderline eigenvalue that is of type-1 if $n=2,4,\ldots$ and of type-2 if $n=3,5,\ldots$.

\noindent (iv) $\rho=-\xi_n$: $\lambda=-n$ is a borderline eigenvalue that is of type-2 if $n=2,4,\ldots$ and of type-1 if $n=3,5,\ldots$.

Depending on the (real) value of $\rho$, Table 1 gives the number of type-1 extraordinary eigenvalues and the number of type-2 extraordinary eigenvalues. 
\end{proposition}
\begin{table}[h!]
\begin{center}
    \begin{tabular}{ | p{3cm} | p{4.5cm} | p{4.5cm} |}
    \hline value of $\rho$
     & number of type-1 \newline extraordinary eigenvalues & number of type-2 \newline extraordinary eigenvalues \\ \hline
    $\rho<-\xi_n$ & $1$ & $1$ \\ \hline
    $-\xi_n \leq \rho < -1$ & $\begin{cases}1, {\rm \,if\,\,} n=2,4,6,\ldots \\ 0, {\rm \,if\,\,} n=3,5,7,\ldots \end{cases}$ & $\begin{cases}0, {\rm \,if\,\,} n=2,4,6, \ldots \\ 1, {\rm \,if \,\,} n=3,5,7,\ldots \end{cases}$ \\ \hline
    $-1 \leq \rho \leq 1$ & $0$ & $0$  \\ \hline
    $1 < \rho \leq \xi_n$ & $0$ & $1$  \\ \hline
    $\rho>\xi_n$ & $1$ & $1$  \\ \hline
    \end{tabular}
\end{center}
    \caption{Numbers (0 or 1) of extraordinary type-1 and extraordinary type-2 eigenvalues for the special case $\rho \in \mathbb{R}$.}\label{tab:specialcaserho}
\end{table}

The present paper will help us view Proposition \ref{prop:realcase} as a corollary of more general results for which $\rho \in \mathbb{C}$. And---as is often the case---the extension into the complex domain will help us better understand and visualize the special case $\rho \in \mathbb{R}$, including certain particularities of this case. {Our figures, for example, will clearly illustrate that  borderline/double eigenvalues can appear} when $\rho \in \mathbb{C}\setminus\mathbb{R}$, but not when $\rho\in \mathbb{R}$. 

\section{Preliminaries}

What follows builds upon results (or corollaries of results) from \cite{Fik2018}, given in this section as lemmas. Our first lemma is Lemma 2.5 of \cite{Fik2018} (throughout this paper, the overbar denotes the complex conjugate){.}

\begin{lemma}\label{lemma:lemma1}
\cite{Fik2018} Let $\rho \in \mathbb{C}$, let $(\lambda,\mathbf{y})$ be an eigenpair of $K_n(\rho)$, and let $J_n$ be the {signature matrix}
\begin{equation*}
J_{n}=\begin{bmatrix}
1 & 0 & \cdots & 0 & 0 \\
0 & -1 & \cdots & 0 & 0 \\
\vdots & \vdots & \ddots & 0 & 0 \\
0 & 0 & \cdots & (-1)^{n-2} & 0 \\
0 & 0 & \cdots & 0& (-1)^{n-1} 
\end{bmatrix}.
\end{equation*}
Then $K_{n}(\bar{\rho})$ and $K_{n}(-\rho)$ possess the eigenpairs $(\bar{\lambda},\bar{\mathbf{y}})$ and $(\lambda,J_n\mathbf{y})$, respectively. 
\end{lemma}

If $\mathbf{y}$ is a symmetric (or skew-symmetric) vector, then $J_n\mathbf{y}$ remains symmetric (or skew-symmetric) if $n$ is odd, but becomes skew-symmetric (or symmetric) if $n$ is even. Lemma \ref{lemma:lemma1} {thus gives the following result.}

\begin{lemma}\label{lemma:lemma2}
For $\rho \in \mathbb{C}$, let $\lambda$ be a type-1 (or type-2) eigenvalue of $K_n(\rho)$. Then\\
(i) $\bar{\lambda}$ is a type-1 (or type-2) eigenvalue of $K_{n}(\bar{\rho})$.\\
(ii) For $n=3,5,7,\ldots$, $\lambda$ is a type-1 (or type-2) eigenvalue of $K_{n}(-\rho)$.\\
(iii) For $n=2,4,6,\ldots$, $\lambda$ is a type-2 (or type-1) eigenvalue of $K_{n}(-\rho)$.
\end{lemma}

Theorems 4.1, 4.5, and 6.5 of \cite{Fik2018} {imply the lemma that follows.}

\begin{lemma}\label{lemma:lemma3} \cite{Fik2018}
Let $\rho \in \mathbb{C}$. $\lambda \in \mathbb{C}$ is a type-1 (or type-2) eigenvalue of $K_n(\rho)$ iff
\begin{equation}\label{eq:lambda}
\lambda=-\frac{\sin(n\mu)}{\sin{\mu}} \quad \left( {\rm or}\,\,\lambda=\frac{\sin(n\mu)}{\sin{\mu}}  \right),
\end{equation}
where $\mu\in \mathbb{C}$ satisfies
\begin{equation}\label{eq:eigen1eq}
\rho=\frac{\sin{\frac{(n+1)\mu}{2}}}{\sin{\frac{(n-1)\mu}{2}}} \quad \left({\rm or}\,\,  \rho=\frac{\cos{\frac{(n+1)\mu}{2}}}{\cos{\frac{(n-1)\mu}{2}}}\right).
\end{equation}

For $\rho \in \mathbb{C}\setminus\{-1,0,1\}$, the $\lambda$ given by (\ref{eq:lambda}) is a repeated type-1 (or repeated type-2) eigenvalue iff, in addition to (\ref{eq:eigen1eq}), $\mu$ satisfies 
\begin{equation}\label{eq:eigen2eq}
\xi_n\cos{\frac{(n+1)\mu}{2}}=\rho \cos{\frac{(n-1)\mu}{2}} \quad \left({\rm or}\,\, \xi_n\sin{\frac{(n+1)\mu}{2}}= \rho \sin{\frac{(n-1)\mu}{2}} \right).
\end{equation}
\end{lemma}
\begin{proof}
Eqns. (\ref{eq:lambda}) and (\ref{eq:eigen1eq}) follow from: (i) Theorem 4.1 of \cite{Fik2018} when $\rho\ne\pm 1, \pm \xi_n$; (ii) Theorem 6.5 of \cite{Fik2018} when $\rho=1$ or $\rho=\xi_n$; and (iii) Lemma \ref{lemma:lemma2} when $\rho=-1$ or $\rho=-\xi_n$ [use (ii) with $\mu+\pi$  in place of $\mu$]. The first equation in (\ref{eq:eigen2eq}) (for the type-1 case) is (4.21) in Theorem 4.5 of \cite{Fik2018}, and the second one is entirely similar. 
\end{proof}

Throughout, we also use the following elementary statements about the Chebyshev polynomial $U_{n-1}(t)$, where $t=\cos x$ or $t=\cosh x${.}
\begin{lemma}\label{lemma:lemmath}
For $n=2,3,4,\ldots$ and $x\in \mathbb{R}$, we have
\begin{equation}\label{eq:t}
\left|\frac{\sin(nx)}{\sin x}\right|\le n,
\end{equation}
and 
\begin{equation}\label{eq:h}
\frac{\sinh(nx)}{\sinh x}\geq n.
\end{equation}
Equality in (\ref{eq:t}) occurs iff $x=0,\pm\pi,\pm 2\pi,\ldots$. In (\ref{eq:h}), equality occurs iff $x=0$.
\end{lemma}

\section{Double eigenvalues}

Ref. \cite{Fik2018} shows that (excluding certain trivial cases) the only repeated eigenvalues of  $K_{n}(\rho)$ are double eigenvalues equal to $-n$. Theorem \ref{th:double} recapitulates this and adds a converse result, namely that (barring exceptional cases) an eigenvalue equal to $-n$ is necessarily a double eigenvalue.

\begin{theorem}\label{th:double}
Let $\rho \in \mathbb{C}\setminus \{-\xi_{n}, -1,0, 1, \xi_{n}\}$ and let $\lambda$ be an eigenvalue of $K_n(\rho)$. Then the following three statements are equivalent:\\
(i) $\lambda$ is a repeated eigenvalue of $K_n(\rho)$.\\
\noindent (ii) $\lambda$ is a double eigenvalue of $K_n(\rho)$.\\
\noindent (iii) $\lambda=-n$.
\end{theorem}

\begin{proof}
(ii) $\implies$ (i) is obvious, while (i) $\implies$ (ii) and (i) $\implies$ (iii) are in Theorem 4.5 of  \cite{Fik2018}. To show (iii) $\implies$ (i), let
$-n$ be a type-1 (or type-2) eigenvalue of $K_n(\rho)$. Then, by Lemma \ref{lemma:lemma3},
\begin{equation}\label{eq:lambdan}
\frac{\sin(n\mu)}{\sin{\mu}}=n \quad \left( {\rm or}\,\,\frac{\sin(n\mu)}{\sin{\mu}}=-n  \right),
\end{equation}
where $\mu$ satisfies \eqref{eq:eigen1eq}. Eqns. \eqref{eq:lambdan} and \eqref{eq:xin} imply the equalities
\begin{equation*}
\frac{\tan\frac{(n-1)\mu}{2}}{\tan\frac{(n+1)\mu}{2}}=\frac{1}{\xi_n}\quad\left({\rm or}\,\,\frac{\tan\frac{(n-1)\mu}{2}}{\tan\frac{(n+1)\mu}{2}}=\xi_n\right),
\end{equation*}
which, when combined with \eqref{eq:eigen1eq}, yield \eqref{eq:eigen2eq}. Therefore the eigenvalue $-n$ is a repeated one by Lemma \ref{lemma:lemma3}, completing our proof.
\end{proof}
\begin{remark}\label{remark: exceptions}
In Theorem \ref{th:double}, all excluded values of $\rho$  are exceptional. The exceptions are as {follows.} When $\rho=\pm\xi_n$, $\lambda=-n$ is an eigenvalue (see Proposition \ref{prop:realcase}) that is non-repeated by Proposition 6.1 of \cite{Fik2018}. When $\rho=\pm 1$, $\lambda=0\ne -n$ is (at least when $n>2$) a repeated eigenvalue, see (2.7) of \cite{Fik2018} or (6.22) of \cite{Fik2018}.  When $\rho=0$, finally, $\lambda=1\ne -n$ is a repeated eigenvalue of the identity matrix $K_n(0)$.
\end{remark}

\begin{corollary}\label{th:criticaldef}
{Let $\rho \in \mathbb{C}$. If $\lambda$ is a borderline/double eigenvalue of $K_{n}(\rho)$, then $\rho \in \mathbb{C}\setminus\mathbb{R}$ and $\lambda=-n$. By way of a converse, an eigenvalue $\lambda$ of $K_{n}(\rho)$ is a borderline/double eigenvalue if any one of the following statements is true:}\\ 
(i) $\lambda$ is a double eigenvalue and $\lambda=-n$; or \\
(ii) $\lambda$ is a repeated borderline eigenvalue; or\\
(iii)  $\lambda=-n$ and $\rho\ne\pm\xi_n$; or\\
(iv) $\lambda$ is a repeated eigenvalue, $\rho\ne\pm 1$, and $\rho\ne 0$.
\end{corollary}

Theorem 4.5 of \cite{Fik2018} allows one to compute (via the solution to a polynomial equation) all $\rho\in\mathbb{C}$ for which $K_n(\rho)$ possesses a type-1 (or type-2) {borderline/double} eigenvalue.

\section{Borderline eigenvalues and closed curves $B_n^{(1)}$, $B_n^{(2)}$}

The theorem that follows is the heart of this paper, as it enables us to compute all complex values $\rho$ that give rise to type-1 (or  type-2) \textit{borderline} eigenvalues. The theorem specifically asserts that the said values of $\rho$ coincide with the range of a complex-valued function $f_n^{(1)}(u)$ [or $f_n^{(2)}(u)$], where $u \in [-\pi, \pi]$. The functions $f_n^{(1)}(u)$ and $f_n^{(2)}(u)$ are defined via the unique solution to a certain transcendental equation.

\begin{theorem}\label{th:bordeigen}
For $\rho \in \mathbb{C}$, $K_{n}(\rho)$ possesses a borderline type-1 eigenvalue $\lambda$ iff $\rho=f_n^{(1)}(u)$ and $\lambda=b_n^{(1)}(u)$ where
\begin{equation}\label{eq:rhosdef}
f_n^{(1)}(u)=\frac{\displaystyle\sin{\frac{(n+1)\mu(n,u)}{2}}}{\displaystyle\sin{\frac{(n-1)\mu(n,u)}{2}}},\qquad b_n^{(1)}(u)=-\frac{\sin[n\mu(n,u)]}{\sin\mu(n,u)},
\end{equation}
in which 
\begin{equation}\label{eq:mupar}
\mu(n,u)=u+i v(n,u).
\end{equation}
In \eqref{eq:mupar}, $u \in [-\pi, \pi]$ is arbitrary while $v(n,u)$ is the function of $u$ that is defined as the unique non-negative root of the transcendental equation
\begin{equation}\label{eq:ba}
\sinh^2(nv)-n^2 \sinh^2{v}=g_{n}(u), \quad v \geq 0,
\end{equation}
in which 
\begin{equation}\label{eq:an}
g_{n}(u)=n^2 \sin^2 {u}-\sin^2(nu), \quad u \in [-\pi, \pi].
\end{equation}
Similarly, $K_{n}(\rho)$ possesses a borderline type-2 eigenvalue $\lambda$ iff $\rho=f_n^{(2)}(u)$ and $\lambda=b_n^{(2)}(u)$ where
\begin{equation}\label{eq:rhocdef}
f_n^{(2)}(u)=\frac{\cos{\displaystyle\frac{(n+1)\mu(n,u)}{2}}}{\cos{\displaystyle\frac{(n-1)\mu(n,u)}{2}}},\qquad b_n^{(2)}(u)=\frac{\sin[n\mu(n,u)]}{\sin\mu(n,u)},
\end{equation}
in which, once again,  $u \in [-\pi, \pi]$ is arbitrary and $\mu(n,u)$ is found from \eqref{eq:mupar}--\eqref{eq:an}. 
\end{theorem}

\begin{proof} 
In Lemma \ref{lemma:lemma3}, take $|\lambda|=n$ and  
set $u={\rm Re}\mu$ and $v={\rm Im}\mu$ to see that $K_{n}(\rho)$ has a borderline type-1 eigenvalue $\lambda$ iff $\rho=\rho_n(u,v)$ and $\lambda=\lambda_n(u,v)$ where
\begin{equation}\label{eq:rho1uv}
\rho_n(u,v)=\frac{\displaystyle\sin{\frac{(n+1)u}{2}}\cosh{\frac{(n+1)v}{2}}+i\cos{\frac{(n+1)u}{2}}\sinh{\frac{(n+1)v}{2}}}
{\displaystyle\displaystyle\sin{\frac{(n-1)u}{2}}\cosh{\frac{(n-1)v}{2}}+i\cos{\frac{(n-1)u}{2}}\sinh{\frac{(n-1)v}{2}}},\quad u,v \in \mathbb{R},
\end{equation}
and
\begin{equation}\label{eq:lambdannn}
\lambda_n(u,v)=-\frac{\sin[n(u+iv)]}{\sin(u+iv)}
,\quad u,v \in \mathbb{R},
\end{equation}
where $u$, $v$, and $n$ are interrelated via
\begin{equation}\label{eq:uv}
n^2=\frac{\sin^2(nu)+\sinh^2(nv)}{\sin^2{u}+\sinh^2{v}}, \quad u,v \in \mathbb{R}.
\end{equation}
Since the right-hand sides of (\ref{eq:rho1uv})--(\ref{eq:uv}) are $2\pi$-periodic in $u$, we assume $u \in [-\pi, \pi]$ with no loss of generality. Since, also,  $\rho_n(-u,-v)=\rho_n(u,v)$ and $\lambda_n(-u,-v)=\lambda_n(u,v)$, we further assume $v \geq 0$. 

Eqn. (\ref{eq:uv}) is then equivalent to the definition \eqref{eq:an} and the transcendental equation \eqref{eq:ba}. This equation has a {\it unique} solution $v=v(n,u) \geq 0$ because: (i) $g_{n}(u)\geq 0$ for $u\in[-\pi,\pi]$; and (ii) the left-hand side of (\ref{eq:ba}) vanishes when $v=0$ and has a positive derivative in $(0,+\infty)$ (assertions (i) and (ii) are readily shown via Lemma \ref{lemma:lemmath}).

With $v=v(n,u)$ thus determined, the $\rho_n(u,v)$  of (\ref{eq:rho1uv}) and the $\lambda_n(u,v)$ of (\ref{eq:lambdannn}) are no longer functions of $v$, and the notations $f_n^{(1)}(u)=\rho_n(u,v)$ and $b_n^{(1)}(u)=\lambda_n(u,v)$ prove (\ref{eq:rhosdef}) with (\ref{eq:mupar}). We have thus shown all assertions pertaining to type-1 eigenvalues. 

For the type-2 case, proceed as before with 
\begin{equation*}
\rho_n(u,v)=\frac{\displaystyle\cos{\frac{(n+1)u}{2}}\cosh{\frac{(n+1)v}{2}}-i\sin{\frac{(n+1)u}{2}}\sinh{\frac{(n+1)v}{2}}}
{\displaystyle\displaystyle\cos{\frac{(n-1)u}{2}}\cosh{\frac{(n-1)v}{2}}-i\sin{\frac{(n-1)u}{2}}\sinh{\frac{(n-1)v}{2}}},\quad u,v \in \mathbb{R},
\end{equation*}
in place of \eqref{eq:rho1uv}, and with the opposite sign in (\ref{eq:lambdannn}). Otherwise, the proof is identical.
\end{proof}

The lemma that follows lists some properties (to be used several times throughout) of the functions of $u$ encountered in Theorem \ref{th:bordeigen}.

\begin{lemma}\label{lemma:properties}
(i) The real and imaginary parts of the functions  $f_n^{(1)}(u)$ and $f_n^{(2)}(u)$ can be found from 
\begin{equation}\label{eq:t1}
f_n^{(1)}(u)=\frac{[\cos u\cosh(nv)- \cos(nu)\cosh v]-i[\sin u\sinh(nv)- \sin(nu)\sinh v]}{\cosh[(n-1)v]-\cos[(n-1)u]},
\end{equation}
\begin{equation}\label{eq:t2}
f_n^{(2)}(u)=\frac{[\cos u\cosh(nv)+ \cos(nu)\cosh v]-i[\sin u\sinh(nv)+ \sin(nu)\sinh v]}{\cosh[(n-1)v]+ \cos[(n-1)u]},
\end{equation}
in which $v$ stands for the $v(n,u)$ of (\ref{eq:ba})--(\ref{eq:an}), for which 
\begin{equation*}
v(n,u)=v(n,-u)=v(n,\pi-u).
\end{equation*}
(ii) Let $k=1$ or $k=2$ and let $u\in (-\pi,\pi]$. Then there is a one-to-one correspondence between $u$ and $\mu(n,u)$. Furthermore,   
\begin{equation}\label{eq:t10}
u=0 \ {\rm{or}} \ u=\pi\Longleftrightarrow v(n,u)=0 \Longleftrightarrow \mu(n,u)\in\mathbb{R}\Longleftrightarrow f_n^{(k)}(u)\in \mathbb{R}.
\end{equation} 
(iii) The functions $v(n,u)$, $\mu(n,u)$, $f_n^{(1)}(u)$, and $f_n^{(2)}(u)$ are $2\pi$-periodic and continuous.\\ 
(iv) The real values $f_n^{(k)}(0)$ and $f_n^{(k)}(\pm\pi)$, and the corresponding values $b_n^{(k)}(0)$ and $b_n^{(k)}(\pm\pi)$ are given by
\begin{equation}\label{eq:rhosdef0}
f_n^{(1)}(0)=\xi_n,\quad f_n^{(2)}(0)=1,
\end{equation}
\begin{equation}\label{eq:rhosdefpi}
f_n^{(1)}(\pm\pi)=
\begin{cases}-1, {\rm \,if\,\,} n=2,4,\ldots \\ -\xi_n, {\rm \,if\,\,} n=3,5,\ldots, \end{cases} \quad
f_n^{(2)}(\pm\pi)=
\begin{cases}-\xi_n, {\rm \,if\,\,} n=2,4,\ldots \\ -1, {\rm \,if\,\,} n=3,5,\ldots, \end{cases}
\end{equation}
and
\begin{equation}\label{eq:rhdef0more}
b_n^{(1)}(0)=-n,\quad b_n^{(2)}(0)=n,
\end{equation}
\begin{equation}\label{eq:rhdefpimore}
b_n^{(1)}(\pm \pi)=(-1)^nn,
\quad
b_n^{(2)}(\pm \pi)=(-1)^{n+1}n.
\end{equation}
(v) For $k=1$ and $k=2$, we have ${\rm Im} f_n^{(k)}(u)<0$ if $0<u<\pi$ and  ${\rm Im} f_n^{(k)}(u)>0$ if $-\pi<u<0$.\\
(vi) The functions $v(n,u)$, $\mu(n,u)$, $f_n^{(1)}(u)$, and $f_n^{(2)}(u)$ are differentiable for $u\in (-\pi,0)$ and $u\in (0,\pi)$, with
\begin{equation}\label{eq:t3}
\frac{dv(n,u)}{du}=-\frac{\sin(2nu)-n\sin(2u)}{\sinh[2nv(n,u)]-n\sinh[2v(n,u)]},
\end{equation}
\begin{equation}\label{eq:t31}
\frac{d\mu(n,u)}{du}=1+i\,\frac{dv(n,u)}{du},
\end{equation}
\begin{equation}\label{eq:t32}
\frac{df_n^{(k)}(u)}{du}=\frac{-{n}\sin[\mu(n,u)]\pm\sin[n\mu(n,u)]}{1\mp\cos[(n-1)\mu(n,u)]}\,\frac{d\mu(n,u)}{du},\quad \quad k=\left\{{1 \atop 2}\right\},
\end{equation}
where the notation means that the upper (lower) sign corresponds to $k=1$ ($k=2$). 
\end{lemma}
\begin{proof} (i) Eqns. (\ref{eq:t1}) and (\ref{eq:t2}) can be shown by manipulating the right-hand sides  of the first equations in (\ref{eq:rhosdef}) and  (\ref{eq:rhocdef}), and setting $\mu(n,u)=u+iv(n,u)$.\\
(ii) follows easily from the definitions in Theorem \ref{th:bordeigen} using Lemma \ref{lemma:lemmath} and Lemma \ref{lemma:properties}(i).\\
(iii) Both $g_n(u)$ and the left-hand side of (\ref{eq:ba}) are continuous. Thus $v(n,u)$---which is the unique solution to (\ref{eq:ba})---is continuous. What remains follows from the definitions in Theorem \ref{th:bordeigen}.\\
(iv) follows from (\ref{eq:rhosdef}), (\ref{eq:rhocdef}), and (\ref{eq:t10}).\\
(v) can be shown via (\ref{eq:t1}) and (\ref{eq:t2}).\\
(vi) Eliminate $g_n(u)$ from (\ref{eq:ba}) and (\ref{eq:an}) to obtain an implicit equation relating $u$ and $v$ and then compute $dv/du$ via the partial derivatives of the implicit function. Eqn. (\ref{eq:t10}) and  Lemma \ref{lemma:lemmath} guarantee that the denominator in (\ref{eq:t3}) does not vanish in $(-\pi,0)$ and $(0,\pi)$. Eqn. (\ref{eq:mupar}) gives (\ref{eq:t31}), while chain differentiation of (\ref{eq:rhosdef}) and (\ref{eq:rhocdef})  gives (\ref{eq:t32}). The denominator in (\ref{eq:t32}) is nonzero because (\ref{eq:t10}), $u\in(-\pi,\pi)$, and $u\ne0$ imply $\mu(n,u)\notin \mathbb{R}$.
 \end{proof}

The following restatement of Theorem \ref{th:bordeigen} introduces the closed curves $B^{(k)}_n$ traced out by $f^{(k)}_n(u)$; they will be referred to as \textit{borderline curves}.

\begin{corollary}\label{th:curves}
The matrix $K_{n}(\rho)$ possesses a borderline type-1 eigenvalue iff $\rho \in B^{(1)}_n$, where $B^{(1)}_n$ is the closed curve given by
\begin{equation}\label{eq:Ls}
B^{(1)}_n=\left\{\rho \in \mathbb{C}: \rho=f^{(1)}_n(u)\text{\,\,for some\,\,} u\in[-\pi,\pi]\right\},
\end{equation}
in which $ f^{(1)}_n(u)$ is defined in (\ref{eq:rhosdef})--(\ref{eq:an}). Similarly, $K_{n}(\rho)$ possesses a borderline type-2 eigenvalue iff $\rho \in B^{(2)}_n$, where $ B^{(2)}_n$ is the closed curve given by
\begin{equation}\label{eq:Lc}
B^{(2)}_n=\left\{\rho \in \mathbb{C}: \rho=f^{(2)}_n(u) \text{\,\,for some\,\,} u\in[-\pi,\pi]\right\},
\end{equation}
in which $f^{(2)}_n(u)$ is defined in (\ref{eq:mupar})--(\ref{eq:rhocdef}).
\end{corollary}

Given $n$ and for $k=1$ or $k=2$, a point $\rho \in B^{(k)}_n$ can be determined as follows{.} Pick $u\in(-\pi,\pi]$; compute $g_{n}(u)$ from (\ref{eq:an}); solve the transcendental equation (\ref{eq:ba}) for $v=v(n,u)$; set $\mu(n,u)=u+iv(n,u)$; find $f_n^{(k)}(u)$ from (\ref{eq:rhosdef}) or (\ref{eq:rhocdef}); and set $\rho=f^{(k)}_n(u)$. Repeat the above process for many $u\in (-\pi,\pi]$ until the continuous curve $B^{(k)}_n$ is depicted. Fig. \ref{fig:twocurves} shows the curves thus generated for $n=5$ and $n=6$.

The properties listed below are apparent in the examples of Fig. \ref{fig:twocurves} and follow easily from Lemma \ref{lemma:lemma2}, or via Theorem \ref{th:bordeigen}. In particular, the points $\rho$ mentioned in (i)-(iv) of Proposition \ref{prop:realcase} are the four intersections of $B^{(1)}_n$ and $B^{(2)}_n$ with the real-$\rho$ axis.

\begin{proposition}\label{th:symmetries}
The borderline curves $B_n^{(1)}$ and $B_n^{(2)}$ exhibit the following properties{.}\\
(i) For $k=1$ and $k=2$, $B_n^{(k)}$ intersects the real axis exactly twice. The two points of intersection are the $f_n^{(k)}(0)$ and $f_n^{(k)}(\pi)$ given in (\ref{eq:rhosdef0}) and (\ref{eq:rhosdefpi}).\\
(ii) Both $B_n^{(1)}$ and $B_n^{(2)}$ are symmetric with respect to the real $\rho$-axis.\\
(iii) The union $B_n^{(1)}\cup B_n^{(2)}$ is symmetric with respect to the origin $\rho=0$.\\
(iv) For $n=3,5,7\ldots$, both $B_n^{(1)}$ and $B_n^{(2)}$ are symmetric with respect to the imaginary $\rho$-axis.\\
(v) For $n=2,4,6\ldots$,  $B_n^{(1)}$ and $B_n^{(2)}$ are mirror images of one another with respect to the imaginary $\rho$-axis.
\end{proposition}

\begin{proof} (i) follows from (\ref{eq:t10}) and Lemma \ref{lemma:properties}(iv). 
The proofs of (ii)-(v) are very similar. Thus we only show (v), which amounts to 
\begin{equation*}
\rho\in B_n^{(1)} \Leftrightarrow -\bar{\rho}\in B_n^{(2)},\quad n=2,4,\ldots.
\end{equation*}
If $\rho\in B_n^{(1)}$, then $K_n(\rho)$ has a type-1 eigenvalue $\lambda$ with $|\lambda|=n$.
By Lemma \ref{lemma:lemma2}(i) and Lemma \ref{lemma:lemma2}(iii), $K_n(-\bar{\rho})$ possesses the type-2 eigenvalue {$\bar{\lambda}$}, and $-\bar{\rho}\in B_n^{(2)}$ follows from {$|\bar{\lambda}|=|\lambda|=n$}. The converse can be shown in the same way. 

For an alternative proof via Theorem \ref{th:bordeigen}, note that (v) is tantamount to 
\begin{equation*}
f_n^{(1)}(\pi-u)=-\overline{f_n^{(2)}(u)},\quad n=2,4,\ldots,
\end{equation*}
which is easily verified using Lemma \ref{lemma:properties}(i).
\end{proof}

In Fig. \ref{fig:twocurves}, $B^{(1)}_n$ and $B^{(2)}_n$ intersect one another $2(n-1)$ times, but neither closed curve exhibits \textit{self-} intersections. {We believe this is true for general $n$.}

\begin{conjecture}\label{th:jordan} The closed curves
$B^{(1)}_n$ and $B^{(2)}_n$ are Jordan curves. In other words, for $k=1$ and $k=2$ we have $f^{(k)}_n(u)\ne f^{(k)}_n(u^{\prime})$ whenever $u\ne u^{\prime}$ ($-\pi<u,u^{\prime}\le\pi$).
\end{conjecture}

Although we were not able to prove Conjecture \ref{th:jordan}, we tested it numerically in a number of ways. For example, in all cases we tried, the difference $f^{(k)}_n(u)-f^{(k)}_n(u^{\prime})$ ($u\ne u^{\prime}$) was nonzero (as expected, the difference became small upon approaching a {borderline/double} eigenvalue). In this manner, we checked our  conjecture directly. We also checked it in other (indirect) ways, to be described.

If Conjecture \ref{th:jordan} is indeed true, borderline eigenvalues of a particular type are {unique, as formulated by the conditional proposition that follows.}

\begin{conditionalproposition}\label{th:jordan1}
Let $\rho\in\mathbb{C}$, let $k=1$ or $k=2$, and assume that 
Conjecture \ref{th:jordan}  is true. Then  $K_n(\rho)$ can possess at most one type-$k$ (simple or double) borderline eigenvalue. 
\end{conditionalproposition}

\begin{proof}
Suppose that $\lambda$ and $\lambda^{\prime}$ are both type-$k$ eigenvalues of $K_n(\rho)$. By Theorem \ref{th:bordeigen} there exist $u$ and $u^{\prime}$ in $(-\pi,\pi]$ such that
\begin{equation*}
\rho=f^{(k)}_n (u),\quad \lambda=b^{(k)}_n(u),\quad \rho=f^{(k)}_n (u^{\prime}), \quad \lambda^{\prime}=b^{(k)}_n(u^{\prime}).
\end{equation*}
It follows from Conjecture \ref{th:jordan} that $u=u^{\prime}$ so that $\lambda=\lambda^{\prime}$. (Lemma \ref{lemma:lemma3} allows for a simple or double eigenvalue $\lambda$.) 
\end{proof}

Conditional Proposition \ref{th:jordan1} was corroborated by many numerical tests.

\begin{figure}
\begin{center}
\includegraphics[scale=0.3]{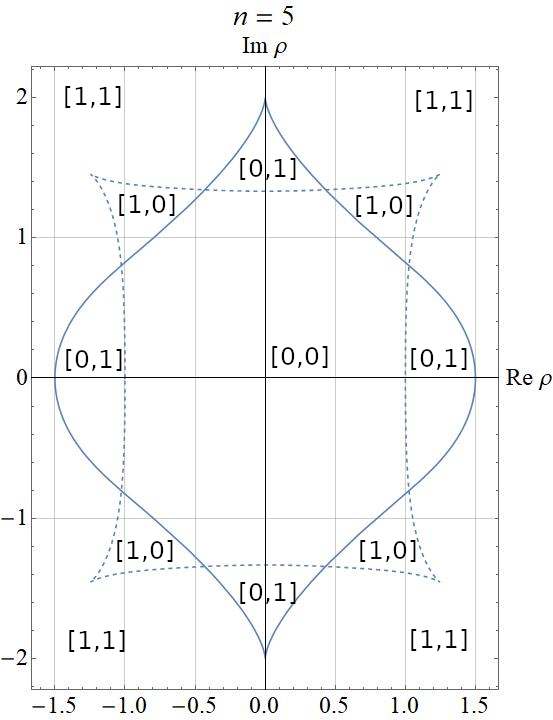}
\includegraphics[scale=0.315]{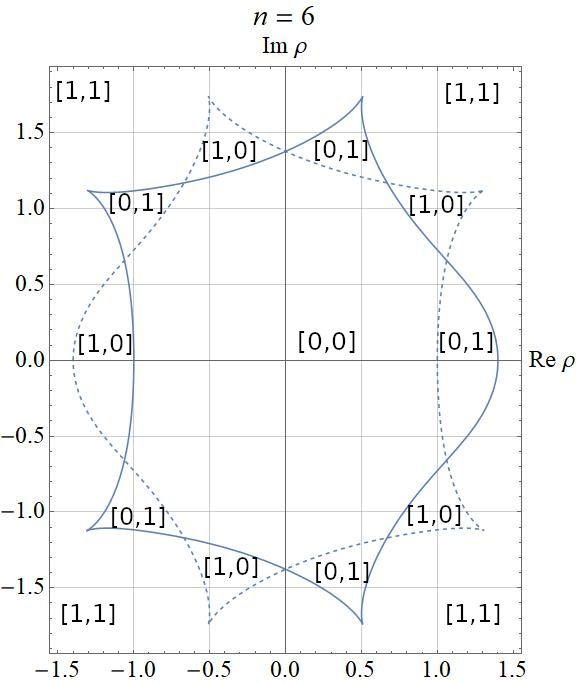}
\end{center}
\caption{Closed curves $B_n^{(1)}$ (solid lines) and $B_n^{(2)}$ (dashed lines) for $n=5$ (left) and $n=6$ (right). The pairs of integers ($[1,0]$, $[0,0]$, etc.) denote the numbers of type-1 and type-2 extraordinary eigenvalues within the regions formed by $B_n^{(1)}$ and $B_n^{(2)}$.} \label{fig:twocurves}
\end{figure}
\section{Double eigenvalues and borderline-curve singularities}

\subsection{Locations of curve singularities}

In Fig. \ref{fig:twocurves}, it is evident that $\rho=i2$ is a singular point of $B^{(1)}_5$, that $B^{(2)}_6$ presents a singularity in each of the four quadrants (the actual values of $\rho$ are $1.31\pm i1.12$ and {$0.51\pm i1.74$}), etc. For any $n$, Theorem \ref{theorem:double1} will enable {\it a priori} determination of \textit{all} singularities in $\mathbb{C}\setminus\mathbb{R}$. 
Our theorem uses the usual definition \cite{Kreyszig} pertaining to parametrized curves, namely that singularities occur whenever $df_n^{(k)}(u)/du=0$. In this manner, Theorem \ref{theorem:double1} will demonstrate a one-to-one correspondence between the aforementioned singularities and {the borderline/double eigenvalues of Definition {\ref{def:borddef2}}.} To begin with, Lemma \ref{lemma:properties}(vi) and $v(n,u)\in\mathbb{R}$ {imply an auxiliary result.}

\begin{lemma}\label{lemma:derivative}
Let $u\in(-\pi,0)\cup (0,\pi)$, let $\mu(n,u)$ be determined via (\ref{eq:mupar})--(\ref{eq:an}), and let $k=1$ or $k=2$. Then  $df^{(k)}_n(u)/du=0$ iff $u$ is such that
\begin{equation}\label{eq:ddouble}
\frac{\sin[n\mu(n,u)]}{\sin\mu(n,u)}=\pm n,\quad \quad k=\left\{{1 \atop 2}\right\}.
\end{equation}   
\end{lemma}

Our theorem follows by translating the conditions on $u$ into conditions on $\rho$ and invoking the definition of a curve singularity{.}

\begin{theorem}\label{theorem:double1}
Let $k=1$ or $k=2$. The point $\rho_0\in \mathbb{C}\setminus \mathbb{R}$ is a singularity of $B^{(k)}_n$ iff $K_n(\rho_0)$ possesses a type-$k$ {borderline}/double eigenvalue.
\end{theorem}

\begin{proof} 
Suppose that $\rho_0\in \mathbb{C}\setminus\mathbb{R}$ is a singular point of $B^{(k)}_n$, so that $\rho_0=f^{(k)}_n(u_0)$ where $d f^{(k)}_n(u_0)/du=0$. Proposition \ref{th:symmetries}(i) gives $\rho_0\ne \pm \xi_n$, $u_0\ne 0$, and $u_0 \ne \pm\pi$. Lemma \ref{lemma:derivative} then implies (\ref{eq:ddouble}). We have thus found a $\mu$ for which equations (\ref{eq:lambda}) and (\ref{eq:eigen1eq}) are satisfied, with the type-$k$ eigenvalue $\lambda$ equal to $-n$. By  {Corollary} \ref{th:criticaldef}(iii), this means that $K_n(\rho_0)$ possesses a type-$k$ {borderline/double} eigenvalue.

{Conversely, suppose that $K_n(\rho_0)$ possesses a type-$k$ borderline/double eigenvalue. Then $\rho_0\notin\mathbb{R}$ and $\lambda=-n$ by Corollary} \ref{th:criticaldef}, so that (\ref{eq:lambda}) and (\ref{eq:eigen1eq}) are satisfied for some $\mu$. As the magnitude of the eigenvalue $\lambda$ is $n$, we must have $\mu=\mu(n,u_0)$ for some nonzero $u_0\in (-\pi,\pi)$, where $\mu(n,u_0)$ is found via (\ref{eq:mupar})-(\ref{eq:an}). Accordingly, (\ref{eq:lambda}) is the same as (\ref{eq:ddouble}). Lemma \ref{lemma:derivative} then gives $df^{(k)}_n(u_0)/du=0$, completing our proof.
\end{proof}

\begin{figure}
\begin{center}
\includegraphics[scale=0.45]{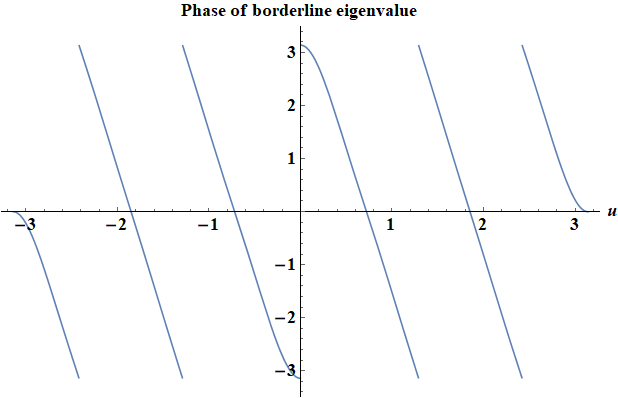}
\end{center}
\caption{Phase (principal value) of type-1 borderline eigenvalue $\lambda=b_6^{(1)}(u)$ as function of $u$ for $n=6$. For $u \ne 0$, the four jumps signal the appearance of a {borderline/double} eigenvalue {equal to $-6$.} The jump at $u=0$ corresponds to a non-repeated eigenvalue equal to $-6$.} \label{fig:phase}
\end{figure}

For $n=6$, Fig. \ref{fig:phase} shows the (principal value of the) phase of the type-1 eigenvalue $\lambda=b_n^{(1)}(u)$ [see (\ref{eq:rhosdef})] as a function of $u$. This is the phase of the borderline eigenvalue as we move along the borderline curve  $B_n^{(1)}$. Discontinuities occur whenever $u=u_0$ is such that the phase jumps by $2\pi$, so that $b_n^{(1)}(u_0)=-n=-6$. By {Corollary} \ref{th:criticaldef}(iii), this means that the borderline eigenvalue becomes a {borderline}/double eigenvalue, with the single exception of the case $u_0=0$, corresponding to $\rho_0=\xi_n=\xi_6=7/5$ [see (\ref{eq:rhosdef0}), Remark \ref{remark: exceptions}, and the solid line in Fig. \ref{fig:twocurves}, right]. When $u_0=0$ the borderline eigenvalue equals $-n$, but {this eigenvalue is not a borderline/double one.} In all other cases, Theorem \ref{theorem:double1} tells us that $\rho_0=f_n^{(1)}(u_0)$ is a singularity of $B_n^{(1)}$. These cases correspond to the four singularities, one in each quadrant, in the solid line in the right Fig. \ref{fig:twocurves}.

\subsection{Cusp-like nature of curve singularities; local bisector}

Let $\rho_0$ be any of the type-1 or type-2 singularities discussed in Theorem \ref{theorem:double1}. Near $\rho=\rho_0$, the two coalescing eigenvalues $\lambda$ can be expanded into a Puiseux series \cite{Lancaster}. The first two terms are
\begin{equation*}
\lambda \cong -n + \eta_0 (\rho-\rho_0)^{1/2},
\end{equation*}
where the square root is double-valued.  Let $\rho-\rho_0=r \, e^{i \theta}$ ($r\ll 1$) and $\eta_0=|\eta_0|\, e^{i\psi}$ so that
\begin{equation*}
|\lambda|^2\cong n^2 + |\eta_0|^2 r - 2n|\eta_0| \sqrt{r}\cos\left(\frac{\theta}{2}+ \psi\right).
\end{equation*}
When $\rho$ is on the level curve $B_n^{(k)}$ we have $|\lambda|^2=n^2$ so that
\begin{equation}\label{eq:cardioideq}
r=r(\theta)\cong\frac{2n^2}{|\eta_0|^2}\left[1 + \cos(\theta + 2\psi)\right],\quad r\ll 1.
\end{equation}
Eq. \eqref{eq:cardioideq} is a polar equation for a cardioid \cite{Lawrence} that has a cusp {at the origin $r=0$, or at $\rho=\rho_0$.} This cardioid is bisected by the ray $\theta=\pi-2\psi$ which is tangent, at $\rho_0$, to the two arcs of the cardioid. Near $\rho=\rho_0$, the cardioid describes the local behavior of $B_n^{(k)}$. We thus use the terms \textit{cusp-like singularity} for any of the $\rho_0$ of Theorem \ref{theorem:double1}, and \textit{local bisector} for the tangent ray passing through $\rho_0$. 

{Since cardioids are Jordan curves, the} above result verifies Conjecture \ref{th:jordan} locally near any cusp-like singularity, i.e., for all $u$ and $u^{\prime}$ such that $f_n^{(k)}(u)$ and $f_n^{(k)}(u^{\prime})$ are sufficiently close to the singularity.

It is possible to determine the $k$-dependent parameters $|\eta_0|$ and $\psi$. Without dwelling on this, we mention that: (i) $B_n^{(1)}$ and $B_n^{(2)}$ each have two cusps on the imaginary axis when $n=5,9,13,\ldots$ and $n=3,7,11,\ldots$, respectively, see the example in the left Fig.~\ref{fig:twocurves}; (ii) the corresponding local bisectors also lie on the imaginary axis. 
\begin{figure}
\begin{center}
\includegraphics[scale=0.35]{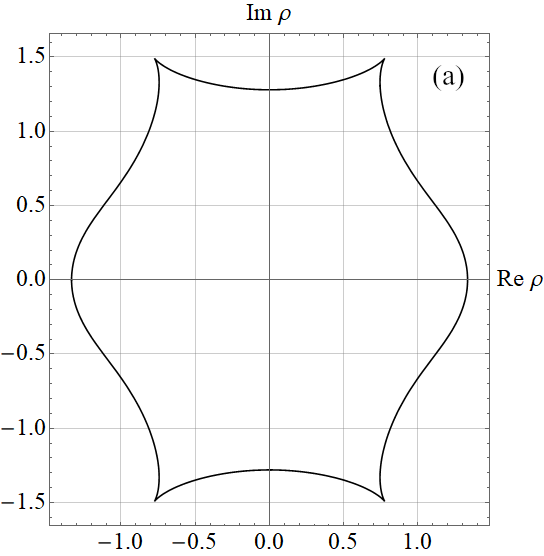}
\includegraphics[scale=0.31]{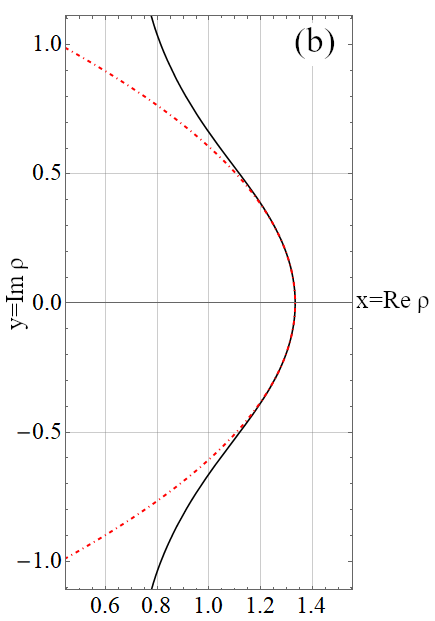}
\end{center}
\caption{{\bf (a)} Closed curve $B_n^{(1)}$ in the $\rho$-plane for $n=7$.  {\bf (b)} Focus on the region around $\xi_7=4/3$ of the curve $B_n^{(1)}$ (solid line). The dot-dashed curve is the parabola of (\ref{eq:b1approxcar}) for $n=7$.} \label{fig:par}
\end{figure}

\subsection{Parabolic behaviors near real axis} 

Recall that there are no {borderline}/double eigenvalues when $\rho\in\mathbb{R}$, and that Sections 5.1 and 5.2 left out the curve-intersections with the real  axis. We thus proceed to Taylor-expand $f^{(k)}_n(u)$ about $u=0$ (formulas for $u=\pi$ then follow from the symmetries listed in Proposition \ref{th:symmetries}). To carry this out, we assume a small-$u$ expansion of $v(n,u)$ of the form
\begin{equation*}
v(n,u)=\alpha u\left[1+\beta u^2+\gamma u^4\right],\quad u\rightarrow 0+0,\quad \alpha>0,
\end{equation*}
substitute into the left-hand side of (\ref{eq:ba}), Taylor-expand the right-hand side $g_n(u)$ using (\ref{eq:an}), equate the coefficients of the resulting powers of $u$ (namely of $u^4$, $u^6$, and $u^8$), and solve for $\alpha$, $\beta$, and $\gamma$. The  result of this procedure is
\begin{equation}\label{eq:vapprox}
v(n,u)=u\left[1-\frac{n^2+1}{15} u^2+\frac{(n^2+1)^2}{150} u^4+O(u^{6})\right],\quad u\rightarrow 0+0.
\end{equation}
In this manner, we have determined an approximation to $v(n,u)$ that satisfies (\ref{eq:ba}) to $O(u^{10})$. The number of terms in (\ref{eq:vapprox}) is sufficient to obtain (nonzero) small-$u$ approximations  to the real and imaginary parts of $f^{(1)}_n(u)-\xi_n$ and $f^{(2)}_n(u)-1$: Substitution of (\ref{eq:vapprox}) into (\ref{eq:rhosdef}) and (\ref{eq:rhocdef}) gives
\begin{equation}\label{eq:ttt1}
f_n^{(1)}(u)=\xi_n\left\{1-\frac{in}{3}u^2-n\left[
\frac{n^2+5n+1}{90}-i\frac{n^2+1}{45}\right]u^4+O(u^{6})\right\},
\quad u\rightarrow 0+0
\end{equation}
and
\begin{equation}\label{eq:ttt2}
f_n^{(2)}(u)=1-inu^2+n\left[
\frac{n^2-5n+1}{10}+i\frac{n^2+1}{15}\right]u^4+O(u^{6}),
\quad u\rightarrow 0+0.
\end{equation}

Let $x={\rm Re} \rho$ and $y={\rm Im} \rho$.  Consistent with Lemma \ref{lemma:properties}(v), the approximations to $y$ in (\ref{eq:ttt1}) and (\ref{eq:ttt2}) are negative, corresponding to the lower-half plane. By Proposition \ref{th:symmetries}(ii), we can obtain an upper-half plane approximation to $y$ by using the opposite sign. Consequently, our final parametrized (small-$u$) results for the behaviors of $B_n^{(k)}$ near the positive real semi-axis are
\begin{equation}\label{eq:b1approxpar}
B_n^{(1)}:\quad x-\xi_n\simeq -\xi_n \frac{n(n^2+5n+1)}{90}u^4,\quad y \simeq \pm \xi_n\frac{n}{3}u^2,\qquad \rho\rightarrow \xi_n,
\end{equation}
\begin{equation}\label{eq:b2approxpar}
B_n^{(2)}:\quad x-1\simeq \frac{n(n^2-5n+1)}{10}u^4,\quad y \simeq \pm nu^2,\qquad \rho\rightarrow 1,
\end{equation}
where, as already mentioned, the upper signs correspond to $u<0$. In Cartesian coordinates, it follows that our curves \textit{locally behave like parabolas} according to
\begin{equation}\label{eq:b1approxcar}
B_n^{(1)}:\quad y^2\simeq \frac{10n\xi_n }{n^2+5n+1}(\xi_n-x)\qquad x\rightarrow \xi_n-0,
\end{equation}
\begin{equation}
B_n^{(2)}:\quad y^2\simeq 
\begin{cases}
\displaystyle\frac{10n}{|n^2-5n+1|}(1-x),\quad n=2,3,4;\quad x\rightarrow 1-0\\
\displaystyle\frac{10n}{n^2-5n+1}(x-1),\quad n=5,6,7,\ldots\quad x\rightarrow 1+0.
\end{cases}
\end{equation}
Fig. \ref{fig:par} depicts representative numerical results.

\section{Extraordinary eigenvalues}

This section deals with extraordinary type-1 (or type-2) eigenvalues. Some of our results assume that that Conjecture \ref{th:jordan} is true, while others do not. 

\subsection{On the number of extraordinary eigenvalues}

If the closed curves  $B_n^{(1)}$ and $B_n^{(2)}$ are indeed Jordan (Conjecture \ref{th:jordan}), then each curve separates the complex-$\rho$ plane into two open and pathwise-connected components, namely an interior and an exterior. 

\begin{conditionalproposition}\label{th:extraordinary}
Assume that 
Conjecture \ref{th:jordan} is true, so that $B_n^{(k)}$ possesses an interior $I_n^{(k)}$ and an exterior $E_n^{(k)}$ ($k=1,2$).
Let $\rho\in\mathbb{C}$,  and let $j_n^{(k)}(\rho)$ denote the number, counting multiplicities, of type-$k$ extraordinary eigenvalues of $K_n(\rho)$. Then all extraordinary eigenvalues are non-repeated, and 
\begin{equation}\label{eq:numberee}
j_n^{(k)}(\rho)=
\begin{cases}
0 , {\rm \,if\,\, } \rho\in I_n^{(k)} \\
0 , {\rm \,if\,\, } \rho\in B_n^{(k)} \\
1, {\rm \,if\,\, } \rho\in E_n^{(k)}   
\end{cases}.
\end{equation}
\end{conditionalproposition}
\begin{proof}
Imagine moving along any (continuous) path in the complex $\rho$-plane. Along the path, the eigenvalues of $K_n(\rho)$ vary continuously. Consequently, the integer-valued function $j_n^{(k)}(\rho)$ can be discontinuous at $\rho=\rho_0$ only if some eigenvalue of $K_n(\rho_0)$ is borderline, i.e., some eigenvalue's magnitude assumes the value $n$ when $\rho=\rho_0$.  According to Corollary \ref{th:curves}, this can happen only if $\rho_0\in B_n^{(k)}$. Therefore, $j_n^{(k)}(\rho)$ remains unaltered along any path lying entirely within $I_n^{(k)}$. Since any two points in $I_n^{(k)}$ can be joined by such a path (lying entirely within $I_n^{(k)}$), $j_n^{(k)}(\rho)$ is constant within $I_n^{(k)}$. Similarly, $j_n^{(k)}(\rho)$ is constant within $E_n^{(k)}$. We thus proceed to find $j_n^{(k)}(\rho)$ for one point $\rho$ within $I_n^{(k)}$ and for one point $\rho$ within $E_n^{(k)}$. 

By Proposition \ref{th:symmetries}(i), the positive real semi-axis intersects $B_n^{(1)}$ at $\rho=\xi_n$ and at no other point; and it intersects $B_n^{(2)}$ at $\rho=1$ and at no other point. We have thus found a point $\rho$ in each region,
\begin{equation}\label{eq:ee1}
0\in I_n^{(1)},\qquad 0\in I_n^{(2)},\qquad \xi_n+1\in E_n^{(1)},\qquad 2\in E_n^{(2)},
\end{equation}
and Table 1 of Proposition \ref{prop:realcase} gives the corresponding $j_n^{(k)}(\rho)$ as
\begin{equation}\label{eq:ee2}
j_n^{(1)}(0)=0,\qquad j_n^{(2)}(0)=0,\qquad
 j_n^{(1)}(\xi_n+1)=1,\qquad j_n^{(2)}(2)=1.
\end{equation}
Both for $k=1$ and $k=2$, we have now shown  that  $j_n^{(k)}(\rho)=0$ and $j_n^{(k)}(\rho)=1$ whenever $\rho\in I_n^{(k)}$ and $\rho\in E_n^{(k)}$, respectively.
 
Now consider a path that lies entirely within  $I_n^{(k)}$, with the single exception of an endpoint that lies on $B_n^{(k)}$. Such a path will also leave $j_n^{(k)}(\rho)$ unaltered, because borderline eigenvalues are ordinary eigenvalues by definition. Thus $j_n^{(k)}(\rho)=0$ for all $\rho\in B_n^{(k)}$. 

The integer $j_n^{(k)}(\rho)$, which equals the number of type-$k$ extraordinary eigenvalues of $K_n(\rho)$, counts double type-$k$ eigenvalues twice by definition.  As $j_n^{(k)}(\rho)$ is at most 1, all extraordinary eigenvalues are non-repeated.
\end{proof}

\begin{remark}\label{remark: eeremark}
Suppose that Conjecture \ref{th:jordan} is not true, and that there exist $B_n^{(k)}$ that separate the plane into more than two components, with all components open, disjoint, and pathwise connected. For any such $B_n^{(k)}$, let $E_n^{(k)}$ denote the component that extends to infinity and let $O_n^{(k)}$ denote the component that includes the origin. Then it is still true that
\begin{equation}\label{eq:numbereenojordan}
j_n^{(k)}(\rho)=
\begin{cases}
0 , {\rm \,if\,\, } \rho\in O_n^{(k)} \\
1, {\rm \,if\,\, } \rho\in E_n^{(k)},   
\end{cases} 
\end{equation}
where, as before, $j_n^{(k)}(\rho)$ denotes the number of extraordinary eigenvalues of $K_n(\rho)$. 
To prove (\ref{eq:numbereenojordan}), it is only necessary to replace
$I_n^{(k)}$ by $O_n^{(k)}$, $\xi_n+1$ by $\rho^{(1)}$, and $2$ by $\rho^{(2)}$ in the relevant part of the proof of Conditional Proposition \ref{th:extraordinary}, where $\rho^{(1)}$ and $\rho^{(2)}$ are sufficiently large positive numbers. Note that (\ref{eq:numbereenojordan}) makes no statement about $B_n^{(k)}$.
\end{remark}

The numerically-generated curves in Fig. \ref{fig:twocurves} intersect, thus forming a number of regions. We have labeled each region by the pair $\left[j_n^{(1)}(\rho),j_n^{(2)}(\rho)\right]$, where $j_n^{(k)}(\rho)$ corresponds to all interior points of the region. For example, the triangularly shaped region in the southeastern part of the left figure is labeled $[1,0]$ because all points $\rho$ interior to this region belong to $E_5^{(1)}$ and $I_5^{(2)}$. For any such $\rho$, $K_5(\rho)$ has one and only one extraordinary eigenvalue; and this eigenvalue is of type-1. {The labels in Fig.} \ref{fig:twocurves} {can be compared to the numbers (0 or 1) in Table 1 of Proposition} \ref{prop:realcase}.

\subsection{Crossing the borderline}

Let us still assume that Conjecture \ref{th:jordan} is true. Besides those used in the proof of Conditional Proposition \ref{th:extraordinary}, it is instructive to consider paths that start in $I^{(k)}_n$, end in $E^{(k)}_n$, and cross $B^{(k)}_n$ once, say at a point $\rho=C$. Let us for definiteness take $k=1$. Although $j_n^{(1)}(\rho)$ always increases by $1$ as we move past $C$, this can occur in several ways:\\
\\
(i) $C$ is not a cusp-like singularity and $C$ does not belong to $B^{(2)}_n$ (i.e., $C$ does not concurrently belong to the \textit{other} borderline curve): In this case, upon moving past $C$, the magnitude of a type-1 simple eigenvalue exceeds the threshold $n$. By Conditional Proposition \ref{th:jordan1}, there is a unique such eigenvalue. Furthermore, no other eigenvalue, of either type, reaches the threshold.\\
\\
(ii) $C\in B^{(1)}_n\cap B^{(2)}_n$ {(for definiteness, let the path start in $I_n^{(1)}\cap I_n^{(2)}$ and end in  $E_n^{(1)}\cap E_n^{(2)}$):} Now, the magnitudes of \textit{two} simple eigenvalues---one of each type---simultaneously exceed the threshold when we move past $C$,  so that  $[j_n^{(1)},j_n^{(2)}]$ changes from $[0,0]$ to $[1,1]$. This case can be visualized via the examples in Fig. \ref{fig:twocurves}.\\
\\
(iii) $C$ is a cusp-like singularity, corresponding to a {borderline/double} type-1 eigenvalue. (The crossing path can, for example, be the local bisector discussed in Section 5.2.) The fact that $j_n^{(1)}(\rho)$ increases by one [rather than two, see (\ref{eq:numberee})] implies that only one eigenvalue crosses the threshold. This must mean that a second eigenvalue reaches the threshold, but then turns back without crossing. This situation, which is reminiscent of bifurcations in a number of physical problems---see, for example, the \textit{exceptional points} of \cite{science}, \cite{prb}, \cite{capolino}---is illustrated for two cases in Fig. \ref{fig:bf795}. 

In the top figure [Fig. \ref{fig:bf795}(a)], a fine zoom (not shown here for brevity) revealed that the line to the left of the cusp-like singularity (i.e., the seemingly single line corresponding to $d<0$) is actually two lines, corresponding to two different eigenvalues of nearly equal magnitudes. 

Fig. \ref{fig:bf795}(b) is like Fig. \ref{fig:bf795}(a) except that $n$ is much larger (95 instead of 7), $C$ is purely imaginary, and $C$ belongs to $B_n^{(2)}$ rather than  $B_n^{(1)}$. This time, movement along the local bisector means that we are on the imaginary axis which, by Proposition \ref{th:symmetries}(iv), bisects the entire curve $B_n^{(2)}$. When $\rho$ is purely imaginary and $n=3,5,\ldots$, it is a corollary of Lemma \ref{lemma:lemma2} that the type-2 eigenvalues of $K_n(\rho)$ either are real, or come in complex-conjugate pairs (similarly to the roots of a polynomial with real coefficients). In Fig. \ref{fig:bf795}(b) it is evident that both these situations occur: For $d<0$ there are two conjugate eigenvalues while, for $d>0$,  there are two real and unequal eigenvalues. In other words, the single line to the left of $d=0$ represents the exactly-coinciding magnitudes of two complex-conjugate eigenvalues, while the two lines to the right correspond to two real eigenvalues of diverging magnitudes {and of phases initially ($d=0$) equal to $\pi$.} {In fact,  both eigenvalues remain negative for all $d\ge 0$.} If we proceed along the positive imaginary semi-axis [beyond the movement depicted in Fig. \ref{fig:bf795}(b)] we will find that the extraordinary  eigenvalue will asymptotically approach $\rho^{n-1}=-|\rho|^{94}$ and that the ordinary one will tend to the limit $-1$. Eqn. (\ref{eq:lambdaapproximations}) below theoretically verifies these two numerical results.

Apart from Fig. \ref{fig:bf795}, we checked all predictions of Sections 6.1 and 6.2 numerically. Many such tests  can serve as indirect corroborations of Conjecture \ref{th:jordan}.
\begin{figure}
\begin{center}
\includegraphics[scale=0.25]{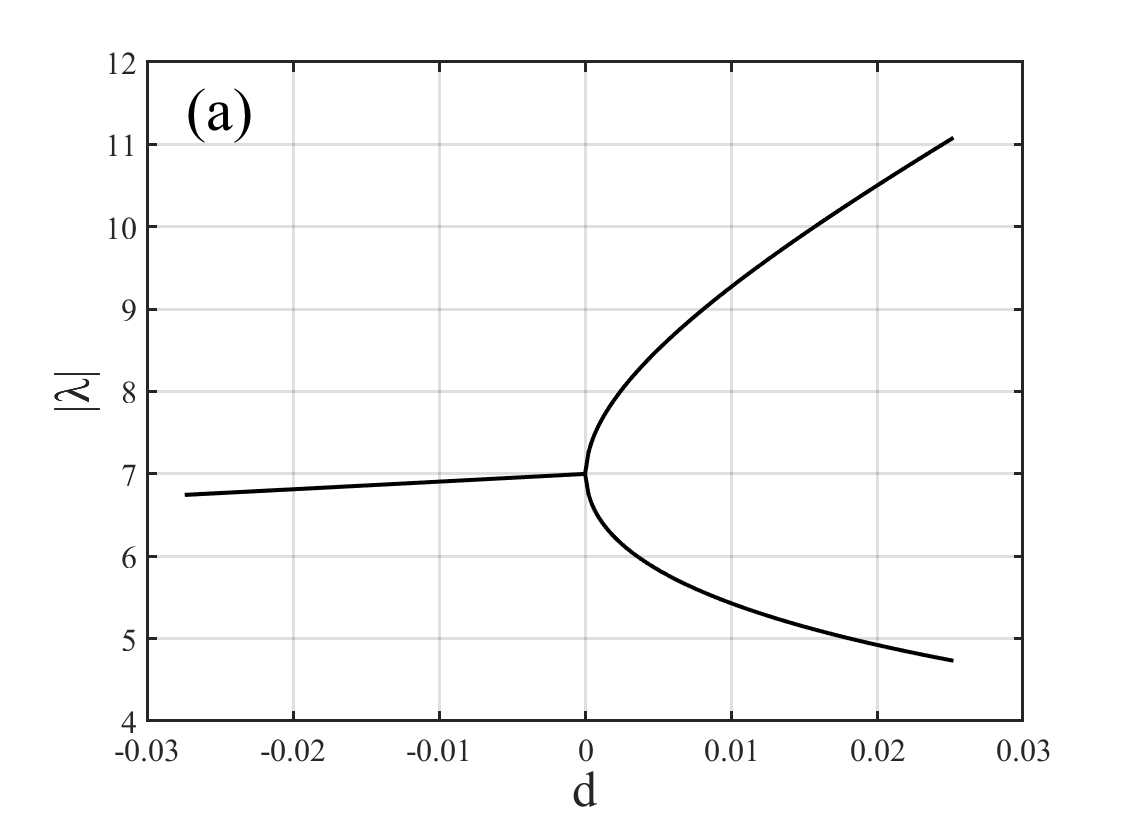}
\includegraphics[scale=0.25]{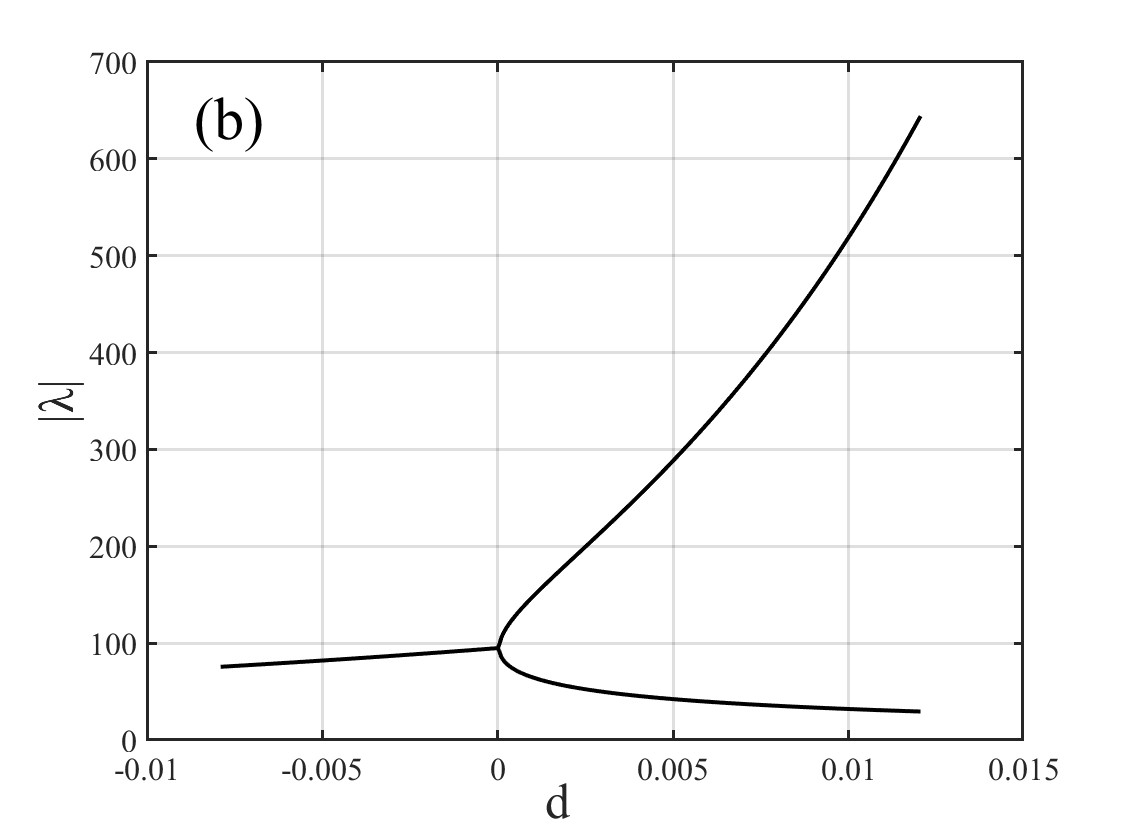}
\end{center}
\caption{{\bf (a)}  Modulus of the two type-1 eigenvalues whose  magnitudes are close to $n$, for $n=7$, as a function of the distance $d$ to the cusp-like singularity at $C=0.77570-i1.49222$. Movement is along the local bisector. The region inside (outside) the curve corresponds to $d<0$ ($d>0$), with $d=0$ at the singularity. {\bf (b)}  As in {\bf (a)}, but with $n=95$, type-2 eigenvalues, and $C=i1.06795$; here, movement is along the imaginary axis, which underlies the local bisector.} \label{fig:bf795}
\end{figure}

\subsection{The limit $|\rho|\rightarrow\infty$} For the case where  
$|\rho|$ is sufficiently large, we now---without assuming Conjecture \ref{th:jordan}---obtain large-$|\rho|$ approximations [eqns. (\ref{eq:lambdaapproximations}) below] to all eigenvalues, valid for any phase of $\rho$. 

Denote the solutions to the two equations in (\ref{eq:eigen1eq}) by $\pm\mu_m$, where an odd (even) index $m$ corresponds to the first (second) equation, corresponding to the type-1 (type-2) case. The $\pm$ sign is irrelevant to the eigenvalues we wish to determine. Via (\ref{eq:eigen1eq}), we can easily justify\footnote{The expressions in (\ref{eq:muapproximations}) can be discovered in a systematic manner by means of the polynomial formulation of Theorem 3.7 of \cite{Fik2018}, whose connection to the equations in (\ref{eq:eigen1eq}) is discussed in Section 4 of   \cite{Fik2018}. The polynomial formulation also assures us that \textit{all} eigenvalues can be found via (\ref{eq:muapproximations}). We finally note that the notations $\mu_m$ and $\lambda_m$ directly correspond to the notations of Theorems 6.5 and 6.6 of \cite{Fik2018} (which deal with the special case $\rho\in\mathbb{R}$) in the following sense: The large-$\rho$ limits ($\rho\in\mathbb{R}$) of the $\mu_m$ and $\lambda_m$ discussed in those theorems coincide with the quantities given in (\ref{eq:muapproximations}) and (\ref{eq:lambdaapproximations}).} the following large-$|\rho|$ approximations: 
\begin{equation}\label{eq:muapproximations}
\mu_m\sim
\begin{cases}
\displaystyle i \ln \rho, {\rm \,if\,\, } m=0\\
\displaystyle i \ln \rho, {\rm \,if\,\, } m=1 \\
\displaystyle \frac{(k-1)\pi}{n-1}, {\rm \,if\,\, } m=2,3,\ldots,n-1   
\end{cases}\quad |\rho|\rightarrow\infty.
\end{equation}
In (\ref{eq:muapproximations}), the large-$|\rho|$ solutions $\mu_2,\mu_3,\ldots,\mu_{n-1}$ correspond to zeros of the denominator of (\ref{eq:eigen1eq}). By contrast, $\mu_0$ and $\mu_1$ result from seeking large-$|\rho|$ solutions such that numerator and denominator have the same order of magnitude. The polynomial formulation of Sections 3 and 4 of \cite{Fik2018} readily ensures that we have found as many solutions to (\ref{eq:eigen1eq}) as are necessary to correspond to \textit{all} eigenvalues.

Denote the corresponding to $\mu_m$ eigenvalue by $\lambda_m$, so that an odd (even) index $m$ corresponds to a type-1 (type-2) eigenvalue. Via (\ref{eq:muapproximations}) and (\ref{eq:lambda}) we obtain
\begin{equation}\label{eq:lambdaapproximations}
\lambda_m\sim
\begin{cases}
\rho^{n-1}, {\rm \,if\,\, } m=0\\
-\rho^{n-1}, {\rm \,if\,\, } m=1 \\
-1, {\rm \,if\,\, } m=2,3,\ldots,n-1   
\end{cases}\quad |\rho|\rightarrow\infty,
\end{equation}
where, for the cases $m=0$ and $m=1$, we have retained the leading term only. Note that (\ref{eq:lambdaapproximations}) predicts $\lambda_0\lambda_1\ldots\lambda_{n-1}\sim(-1)^{n-1}\rho^{2n-2}$, consistent with the fact \cite{Fik2018} that the determinant of $K_n(\rho)$ equals $(1-\rho^2)^{n-1}$. Note also that the large-$n$ formulas for $\lambda_0$ and $\lambda_1$ in Section 5 of \cite{Fik2018} reduce, as $|\rho|\rightarrow\infty$, to the corresponding formulas in (\ref{eq:lambdaapproximations}); this explains why we obtain good numerical agreement even in cases where $\rho$ and $n$ are both large. 

Remark \ref{remark: eeremark} showed that for all sufficiently large $|\rho|$, there are exactly two extraordinary eigenvalues, one of each type. Eqn. (\ref{eq:lambdaapproximations}) further shows that, asymptotically, these eigenvalues ($\lambda_0$ and $\lambda_1$) are equal to plus/minus the largest element of $K_n(\rho)$. Numerically, (\ref{eq:lambdaapproximations}) can give very good results. When, for example, $\rho=15+i12$ and $n=6$, (\ref{eq:lambdaapproximations}) gives three-digit accuracy for the real parts of $\lambda_0$ and $\lambda_1$, and two-digit accuracy for the (smaller) imaginary parts. As for the remaining (ordinary) eigenvalues, the one furthest from $-1$ is approximately $-1.07+i0.05$. Since the matrix elements vary greatly in magnitude when $|\rho|$ is large (this is especially true when $n$ is also large), eqn. (\ref{eq:lambdaapproximations}) can also be useful for numerical computations.

\bibliographystyle{amsplain}

\end{document}